\documentclass[a4paper, 12pt] {article}
\usepackage{amssymb}
\usepackage{amsmath}
\usepackage{amsfonts}
\usepackage{amsthm}
\usepackage{graphicx}
\usepackage{epstopdf}
\usepackage{enumitem}
\usepackage{courier}
\usepackage{adjustbox}
\usepackage{url}
\usepackage{caption}
\usepackage[nottoc,notlot,notlof]{tocbibind}
\newcommand{\nn}{\mathbb{N}}

\newcommand{\cc}{\mathbb{C}}
\newcommand{\pp}{\mathbb{P}}
\theoremstyle{plain}
\newtheorem{thm}{Theorem}[section]
\newtheorem{lemma}[thm]{Lemma}
\newtheorem{cor}[thm]{Corollary}

\theoremstyle{definition}

\newtheorem{defn}{Definition}
\newtheorem{example}[thm]{Example}
\newtheorem{problem}[thm]{Problem}

\usepackage[IL2]{fontenc}
\usepackage{titling}

\title{\bf Generalised Beauville Groups}
\author{\bf L. Carta\thanks{Department of Economics, Mathematics and Statistics, Birkbeck, University of London, Malet Street, London, WC1E 7HX. Email: l.carta@bbk.ac.uk}, B. T. Fairbairn\thanks{Department of Economics, Mathematics and Statistics, Birkbeck, University of London, Malet Street, London, WC1E 7HX. Email: b.fairbairn@bbk.ac.uk ORCiD: https://orcid.org/0000-0003-2871-7719}}
\date{}

\thanksmarkseries{arabic}

\begin{document}
\maketitle

\begin{abstract}
A Beauville group acts freely on the product of two compact Riemann surfaces and faithfully on each one of them. In this paper, we consider higher products and present {\it{generalised Beauville groups}}: for $d \geq 2$, $d$ is the minimal value for which the same action can be defined on the product of $d$ compact Riemann surfaces.
\end{abstract}


\section{Introduction}
\label{sec1b}

\begin{defn}
\label{def1}
Let ${\cal{C}}_1$, ${\cal{C}}_2$ be two compact Riemann surfaces ({\it{i.e.}}, algebraic curves), each of genus $g \geq 2$. Let $G$ be a finite group. Consider the $2$-dimensional variety over the field $\cc$ given by the quotient ${\cal{B}} = ({\cal{C}}_1 \times {\cal{C}}_2) / G$. We call ${\cal{B}}$ a {\it{Beauville surface}} of unmixed type if the action of $G$:
\begin{enumerate}
\item is free on ${\cal{C}}_1 \times {\cal{C}}_2$, that is, without fixed points;
\item is faithful on each ${\cal{C}}_i$ (for $i = 1, 2$) as its group of automorphisms.
\end{enumerate}
Notice that an automorphism of $G$ may give rise to non-isomorphic Beauville surfaces (see \cite{Jon1}).
\end{defn}
Beauville's original construction (see \cite[Exercise X.13 (4)]{Bea}) was first generalised and studied by Catanese in \cite{Cat1} (we will not consider Beauville surfaces of mixed type, where some elements of $G$ transpose the two curves ${\cal{C}}_i$). In the above definition, condition 1. is sometimes referred to as ${\cal{B}}$ being {\it{isogenous to a higher product}}, whereas condition 2. is often expressed in terms of {\it{triangle curves}} (${\cal{C}}_i / G$ is isomorphic to the projective line $\pp ^1 (\cc)$ and the covering ${\cal{C}}_i \rightarrow {\cal{C}}_i / G$ is ramified over $\{ 0, 1, \infty \}$, see \cite[p. 226]{Jon4}). In \cite[Theorem 3.3]{Cat2}, Catanese also proved that Beauville surfaces are {\it{rigid}}, {\it{i.e.}}, they do not admit any nontrivial deformation; if we let $G^0 \subset G$ be the subgroup of $G$ of index $\leq 2$ which does not exchange ${\cal{C}}_1$ and ${\cal{C}}_2$, then rigidity is equivalent to $G^0$ satisfying condition 2. In particular, if we have two complex surfaces ${\cal{B}}$ and ${\cal{B}}'$, where ${\cal{B}}$ is a Beauville surface and $\chi ({\cal{B}}') = \chi ({\cal{B}})$ and $\pi _1 {\cal{B}}' \cong \pi _1 {\cal{B}}$, then ${\cal{B}}$ and ${\cal{B}}'$ are diffeomorphic, that is, isomorphic as differentiable real manifolds. By \cite{Bau1, Gon3, Gon4}, it follows that ${\cal{B}}'$ is biholomorphic to ${\cal{B}}$ or its complex conjugate surface $\bar{{\cal{B}}}$, obtained by applying complex conjugation to the coefficients of the polynomials defining ${\cal{B}}$. Notice that for complex algebraic curves it is exactly the opposite: for each $g > 0$ there are uncountably many isomorphism classes of Riemann surfaces of genus $g$, all mutually homeomorphic, and hence with the same Euler characteristic $\chi = 2 - 2 g$, and with isomorphic fundamental groups.\\

\begin{defn}
\label{def2}
Let $G$ be a finite group. For $x, y \in G$, let
\begin{align*}
\Sigma (x, y) := \bigcup _{i = 1} ^{|G|} \bigcup _{k \in G} \lbrace (x^i)^k , (y^i)^k, ((x y)^i)^k \rbrace.
\end{align*}
That is, $\Sigma$ is the union of the conjugates of the cyclic subgroups generated by $x$, $y$, $x y$, respectively. A set of elements $\{ x_1, y_1, x_2, y_2 \} \subset G$ is an unmixed {\it{Beauville structure}} of $G$ if and only if $\langle x_1, y_1 \rangle = \langle x_2, y_2 \rangle = G$ and
\begin{align*}
\Sigma (x_1, y_1) \cap \Sigma (x_2, y_2) = \{ e \}.
\end{align*}
For $i = 1, 2$, if $x_i, y_i$ determine a Beauville structure of $G$, then we call $G$ a {\it{Beauville group}}.
\end{defn}
Observe that since $\Sigma (x_i, y_i)$ is the set of elements of $G$ with fixed points in ${\cal{C}} _i$ (for $i = 1, 2$), the existence of a Beauville structure implies that the action of $G$ is free on ${\cal{C}}_1 \times {\cal{C}}_2$. Traditionally, authors have added a third element $z_i = (x_i y_i)^{-1}$ to every pair $x_i$, $y_i$ in the above formulation to explicitly reflect condition 2 in Definition \ref{def1} and the consequent rigidity of the corresponding Beauville surface: indeed, $x_i$, $y_i$, $z_i$ represent the local monodromy permutations for the covering ${\cal{C}}_i \rightarrow \pp ^1 (\cc)$ over $\{ 0, 1, \infty \}$ ($F = \langle x, y, z : x y z = 1 \rangle$ is the fundamental group of the projective line minus three points).\\
Beauville structures are determined by taking a $2$-generated finite group and checking for trivial intersections of any {\it{two}} sets $\Sigma$ obtained from pairs of generators. Hence, a natural question is to ask whether we can have $2$-generated finite groups for which there are no Beauville structures, but trivial intersections of {\it{more than two}} sets $\Sigma$ can still be found. Before any attempt to answer that, we need to introduce some auxiliary new concepts.

\begin{defn}
\label{def4}
Let $G$ be a finite group. The minimum number $d$ of sets $\Sigma (x_i, y_i)$ of $G$ ($i = 1, \ldots, d$ and $G = \langle x_i, y_i \rangle$) whose intersection is trivial is called the {\it{Beauville dimension}} of $G$. If there exists at least one non-identity element of $G$ which is contained in every set $\Sigma$, then $d = 1$.
\end{defn}

It could be argued that in the last case $d$ should be left undefined. However, in order to determine the Beauville dimension of a generalised Beauville group, we typically go through a series of consecutive natural numbers until the actual value is reached, and thus starting from (or ending at) 1 seemed to be the obvious choice.

\begin{example}
Any Beauville group has Beauville dimension $d = 2$.
\end{example}

\begin{defn}
\label{def5}
Let $G$ be a finite group with Beauville dimension $d \geq 2$. Then, $G$ is a {\it{generalised Beauville group}}. The corresponding {\it{generalised Beauville structures}} are obtained by considering the $d$ generating pairs of elements. Similarly, the resulting {\it{generalised Beauville varieties}} are just the quotients of the products of $d$ algebraic curves by $G$; they are rigid for the same reasons given for Beauville surfaces.
\end{defn}

\begin{example}
\label{example2}
Consider the group $G \cong C_3 \times C_3$.
Let $x$ and $y$ be a generating pair for $G$. Then, the non-trivial elements of the four proper subgroups of $G$ are as follows:
\begin{enumerate}
\item $A = \{x, x^2\}$,
\item $B = \{y, y^2\}$,
\item $C = \{xy, (xy)^2\}$,
\item $D = \{xy^2, x^2y\}$.
\end{enumerate}
Hence, from each set we get a cyclic subgroup of order 3. Now, applying Definition \ref{def2}, we have that
\begin{enumerate}
\item $\Sigma (x, y) = \{ e \} \cup A \cup B \cup C$,
\item $\Sigma (x, xy) = \{ e \} \cup A \cup C \cup D$,
\item $\Sigma (y, xy^2) = \{ e \} \cup A \cup B \cup D$,
\item $\Sigma (y, xy) = \{ e \} \cup B \cup C \cup D$.
\end{enumerate}
As we can easily notice, only the identity element appears in all four sets $\Sigma$, whereas the intersection of any two or three of them is non-trivial.
Thus, $d(G) = 4$. In particular, $C_3 \times C_3$ is the generalised Beauville group of smallest size (see the table in Section \ref{sec3}).
\end{example}
The above example has been known for some time and was brought to the authors' attention by Gareth A. Jones. Starting from that, we wrote algorithms, implemented in GAP (see \cite{gap}), to search for further cases of Beauville groups and structures.

\begin{example}
\label{example3}
Consider the group
\[{G = \langle x,\,y\,|x^9,\,y^9, [x, y]^3, [x^3,y],\,[x,y^3],\,(xy^2)^3,\,(x^2y)^3,(yx)^3x^3y^3\rangle}\]

({\it{i.e.}},
SmallGroup(243, 4) in GAP notation) with the following generating pairs:
\begin{enumerate}
\item $A = \{ x^{- 1} y^3, x^2 y \}$,
\item $B = \{ x^ 3 y^{- 1}, x y^{- 3} \}$,
\item $C = \{ x^{- 1} y, x y \}$.
\end{enumerate}
As it can be easily checked, $G$ is not a Beauville group since any two of $A$, $B$, $C$ have corresponding sets $\Sigma$ with non-trivial intersections. However, $\Sigma(A) \cap \Sigma(B) \cap \Sigma(C) = \{ e \}$. Hence, $d(G) = 3$. This is the first known example of a generalised Beauville group with Beauville dimension equal to $3$. Among all groups of order $243$, the same value works for $H =$ SmallGroup(243, 13). Thus, $G$ and $H$ are the smallest groups with $d = 3$.
\end{example}

We conclude this section by pointing out a key difference between Beauville groups and generalised Beauville groups with $d > 2$. By Definition \ref{def2}, we know that finding a Beauville structure for a group $G$ is sufficient to establish that $G$ is a Beauville group. Nevertheless, a further distinction needs to be made for generalised Beauville structures obtained from more than $2$ pairs of generators.

\begin{defn}
\label{def5c}
Let $G$ be a finite group. For $n > 2$, let ${\cal{S}} = \{ \Sigma_1, \ldots, \Sigma_n \}$ be a generalised Beauville structure for $G$.
\begin{enumerate}[label=\alph*)]
\item If there exists ${\cal{S'}} \subset {\cal{S}}$ such that ${\cal{S'}}$ is also a generalised Beauville structure for $G$, then ${\cal{S}}$ is a {\it{derived}} structure.
\item If there is no ${\cal{S'}} \subset {\cal{S}}$ such that ${\cal{S'}}$ is also a generalised Beauville structure for $G$, then ${\cal{S}}$ is a {\it{non-derived}} structure.
\item If ${\cal{S}}$ is non-derived and there is no smaller generalised Beauville structure for $G$, then ${\cal{S}}$ is a {\it{minimal}} structure.
\end{enumerate}
\end{defn}

Therefore, by Definition \ref{def4}, only minimal structures determine the Beauville dimension of a group.\\

Surveys discussing related geometric and topological
matters are given by Bauer, Catanese and Pignatelli in \cite{BCP1,BCP2}. Other notable
works in the area include \cite{B,F15,JSurv,CombSurv,Wo}.\\

Section \ref{sec2} contains our results and a complete classification of abelian finite groups with respect to their Beauville dimension. Tables displaying all generalised Beauville groups of orders from 1 to 1023 can be found in Section \ref{sec3}. Open problems and further lines of research conclude our study.

\section{Results}
\label{sec2}
In this section, we prove a number of fundamental results which provide the necessary means to start classifying generalised Beauville groups and to construct further examples whose orders are way beyond our computational constraints.

\begin{thm}
\label{thm1}
Let $G$ and $H$ be finite groups such that $\gcd(|G|, |H|) = 1$. For $n > 2$, let $d(G) = n$, and let $H$ be a Beauville group, that is, $d(H) = 2$. Then, the direct product $P = G \times H$ has $d(P) = n$.
\end{thm}

\begin{proof}
First of all, the fact that $\gcd(|G|, |H|) = 1$ guarantees that $P$ too can be $2$-generated. Since $d(G) = n$, without loss of generality, we can find generating pairs $\{ x_1, y_1 \}$, $\{ x_2, y_2 \}$, $\ldots$, $\{ x_n, y_n \}$ for $G$ such that $\Sigma (x_1, y_1) \cap \Sigma (x_2, y_2) \cap \ldots \cap \Sigma (x_n, y_n)$ is trivial. Similarly, as $H$ is a Beauville group, we can find generating pairs $\{ u_1, v_1 \}$, $\{ u_2, v_2 \}$ such that $\Sigma (u_1, v_1) \cap \Sigma (u_2, v_2)$ is trivial. Now, since $\gcd(|G|, |H|) = 1$, for $g_1, g_2 \in G$ and $h_1, h_2 \in H$ we can find $p, q \in P$ such that $p = (g_1, h_1)$, $q = (g_2, h_2)$ and $p q = (g_1 g_2, h_1 h_2)$. Then, $\Sigma (p, q) = \bigcup (g, h)$ for all $g \in (g_1 ^i)^G \cup (g_2 ^i)^G \cup ((g_1 g_2)^i)^G$ for $1 \leq i \leq |G|$, and all $h \in (h_1 ^j)^H \cup (h_2 ^j)^H \cup ((h_1 h_2)^j)^H$ for $1 \leq j \leq |H|$. Hence, $\Sigma (p, q)$ contains $\Sigma (g_1, g_2)$ and $\Sigma (h_1, h_2)$. In particular, without loss of generality, consider $\Sigma (p_1, q_1)$, which contains $\Sigma (x_1, y_1)$ and $\Sigma (u_1, v_1)$, and $\Sigma (p_2, q_2)$, which contains $\Sigma (x_2, y_2)$ and $\Sigma (u_2, v_2)$. Clearly, the elements belonging to their intersection are of the form $(g, h)$, or $(g, e)$, or $(e, h)$ for any $g \in \Sigma (x_1, y_1) \cap \Sigma (x_2, y_2)$ and $h \in \Sigma (u_1, v_1) \cap \Sigma (u_2, v_2)$. However, $H$ is a Beauville group, and so $(e, h) = (e, e)$ and $(g, h) = (g, e)$. Thus, $\Sigma (p_1, q_1) \cap \Sigma (p_2, q_2) \cong \Sigma (x_1, y_1) \cap \Sigma (x_2, y_2)$. That is, non-empty intersections of sets $\Sigma$ in $G$ map to isomorphic non-empty intersections of sets $\Sigma$ in $P$, which also ensures that we cannot have any values of $d(P)$ smaller than $n$. Therefore, $d(P) = n$.
\end{proof}

\begin{example}
\label{example4}
Looking at the table in Section \ref{sec3}, we can easily spot SmallGroup(225, 6) being the direct product of SmallGroup(9, 2) and SmallGroup(25, 2), or SmallGroup(675, 14) being the direct product of SmallGroup(25, 2) and SmallGroup(27, 3). Moreover, by \cite{Fue}, the Suzuki group $Sz(q)$ (where $q$ is an odd power of $2$) is a Beauville group. Hence, if we set $H = Sz(q)$ in Theorem \ref{thm1} and take the direct product between $H$ and SmallGroup(243, 4), or SmallGroup(9, 2), then we obtain an infinite source of non-soluble generalised Beauville groups with $d = 3$, or $d = 4$, respectively.
\end{example}

We now consider {\it{lifts}} of generalised Beauville structures by extending a result proved in \cite[Lemma 4.2]{Fue}.

\begin{defn}
\label{def6}
Let $G$ be a group and $N$ a normal subgroup of $G$.  An element $g \in G$ is {\it{faithfully represented}} in $G / N$ if $\langle g \rangle \cap N = 1$, or, equivalently, if $g$ has the same order in $G$ as its image in $G / N$. Accordingly, a triple in $G$ is faithfully represented in $G / N$ if each of its elements is faithfully represented in $G/N$.
\end{defn}

\begin{lemma}
\label{lem2}
Let $G$ be a finite group with generating triples $(x_i, y_i, z_i)$ with $x_iy_iz_i = 1$ for $1 \leq i \leq n$, and a normal subgroup $N$ such that at least one of these triples is faithfully represented in $G / N$. If the images of these triples correspond to a generalised Beauville structure for $G/N$, then these triples correspond to a generalised Beauville structure for $G$.
\end{lemma}

\begin{proof}
Assume, without loss of generality, that the triple $(x_1, y_1, z_1)$ is faithfully represented in $G / N$ and that $x_1 ^j$ is conjugate in $G$ to a power of $x_i$, $y_i$, or $z_i$, for every $i$. Since the quotient map is a homomorphism, the image of $x_1 ^j$ in $G / N$ is conjugate in $G / N$ to a power of the image of $x_i$, $y_i$, or $z_i$, for every $i$. Now, $G / N$ is a generalised Beauville group, so this image can only be the identity (which is only conjugate to itself, so also the image of $x_1 ^j$ is the identity). Hence, $x_1 ^j \in N$ (as elements of $G$ which map to the identity of $G / N$ belong to $N$). Moreover, by Definition \ref{def6}, $\langle x_1 ^j \rangle \cap N = 1$, which implies that $x_1 ^j = 1$. Alternatively, by the same definition, $x_1 ^j$ needs to have the same order in $G$. Therefore, there exists a generalised Beauville structure for $G$.
\end{proof}

\begin{lemma}
\label{lem2b}
Let ${\cal{S}}$ be a generalised Beauville structure with $n$ pairs of generators for a group $G$ ($n \geq 2$). Then, $2 \leq d(G) \leq n$.
\end{lemma}

\begin{proof}
If $n = 2$, then, by Definition \ref{def2}, $d(G) = 2$, and so $G$ is a Beauville group. Otherwise, if $n > 2$, then we have to consider the three cases described in Definition \ref{def5c}. If ${\cal{S}}$ is minimal, then $d(G) = n$, whereas if ${\cal{S}}$ is either a derived or a non-derived structure, then $2 \leq d(G) \leq n$. Thus, combining these observations, we get that $2 \leq d(G) \leq n$.
\end{proof}

\begin{cor}
\label{cor3}
Let $G$ be a finite group which satisfies the conditions of Lemma \ref{lem2}. Then, $2 \leq d(G) \leq n$.
\end{cor}

\begin{example}
\label{example5a}
If we remove the relation $(yx)^3x^3y^3$ from the presentation of SmallGroup(243, 4) given in Example \ref{example3}, we obtain SmallGroup(729, 9). Since any minimal structure for the first group can be lifted to a minimal structure for the second one, both groups have the same Beauville dimension $d = 3$.
\end{example}

\begin{example}
\label{example5}
Consider the groups $G = $ SmallGroup(729, 34) and $H = $ SmallGroup(81, 9), with the following presentations:
\begin{align*}
G = \langle x, y : x^3 &= y^3 = (y^2 x y^2 x^2)^3 = (y^2 x^2 y x^2)^3\\
                                 &= (y^2 x^2 y x)^3 = (y^2 x)^2 y x (y x^2)^2 (y x)^2 = 1 \rangle,
\end{align*}
\begin{align*}
H = \langle s, t : s^3 = t^3 = (s t)^3 = s t^2 (s t^2 s^2 t)^2 s^2 t = 1 \rangle.
\end{align*}
Observe that $H = G / N$, where $N \cong C_3 \times C_3$ is the normal subgroup of $G$ generated by $(y x)^3$ and $y x y^2 x (y x^2)^3 y^2 x y$. Now, let $v = (t^2 s)^2$ and $w = t s (t s^2)^2$. The quotient map $f : G \rightarrow H$ sends $x$ to $w$ and $x^2 y$ to $v$. As it can be easily verified, the orders of $x$, $x^2 y$, $y^2$ in $G$ are the same as the orders of $w$, $v$, $(w v)^{-1}$ in $H$: $3$, $9$, $3$, respectively. Hence, the triple $T = \{ x, x^2 y, y^2 \}$ is faithfully represented in $H$ by the triple $T' = \{ w, v, (w v)^{-1} \}$. Moreover, there exists at least one generalised Beauville structure ${\cal{S}}$ for $H$ which has $T'$ as one of its four generating sets:
\begin{align*}
A' &= \Sigma (s, t),\\
B' &= \Sigma (t, w),\\
C' &= \Sigma (t, v),\\
D' &= \Sigma(T').
\end{align*}
By Lemma \ref{lem2}, we know there is a generalised Beauville structure ${\cal{S'}}$ with four generating sets for $G$ involving $T$ as one of them. An example is
\begin{align*}
A &= \Sigma(x, y),\\
B &= \Sigma(x, x y),\\
C &= \Sigma(y, x^2 y),\\
D &= \Sigma(T).\\
\end{align*}
However, $d(H) = 4$, whereas $d(G) = 2$. Thus, by ``lifting" a minimal structure, we have obtained a non-derived one. This confirms that if we are not exclusively dealing with Beauville groups, then Lemma \ref{lem2} can only provide lower and upper bounds, not the exact value of $d$.
\end{example}

Now, in order to classify all finite abelian generalised Beauville groups by proving the main theorem of this section, we will use the following result from \cite[Theorem 3.4]{Bau1} and \cite[Theorem 11.1]{Jon4}.

\begin{thm}
\label{thm2}
Let $G$ be a finite abelian group. Then, $G$ is a Beauville group if and only if $G \cong C_n \times C_n$ where $n > 1$ and $\gcd(n, 6) = 1$.
\end{thm}

The crucial part is then to show that for a finite abelian generalised Beauville group $G$ $d(G) \in \{ 2, 4 \}$.

\begin{thm}
\label{thm3}
Let $H$ be a finite abelian group with $d(H) > 2$. Then, $d(H) = 4$ and $ H = C_n \times C_n \times G$, where $n$ is a power of $3$ and $G$ is either the trivial group, or an abelian Beauville group.
\end{thm}

\begin{proof}
Suppose $d(H) > 2$. If $H \ncong C_n \times C_n$, then, by Theorem \ref{thm2}, $H$ cannot be a Beauville group. The result is proved in \cite[Theorem 11.1]{Jon4} by showing that the intersection of all possible sets $\Sigma$ is non-trivial, which implies that $H$ is not a generalised Beauville group either. So, let $H \cong C_n \times C_n$. If $n$ is even, then it is not difficult to verify that $d(H) = 1$. Hence, by this observation and Theorem \ref{thm2}, we may assume that $n$ is an odd multiple of $3$. If $n = 3^k$, then $H$ has exactly $8$ elements of order $3$ and precisely $6$ of them are contained in each set $\Sigma$ ($C_{3^k}$ has $2$ elements of order $3$, namely $g^l$ and $g^{2 l}$, for $c$ generating the group and $l = 3^{k - 1}$); accordingly, given that $4$ is the smallest number of sets $\Sigma$ we need to have a trivial intersection, $d(H) = 4$ (see Example \ref{example2}). Otherwise, $n = 3^k r$ for some $r$ with $\gcd(r, 6) = 1$. Thus, $H = C_{3^k} \times C_{3^k} \times G$, where $G$ is an abelian Beauville group of order $r^2$ ({\it{i.e.}}, $C_r \times C_r$). By Theorem \ref{thm1}, $d(H) = 4$. This completes the proof.
\end{proof}

From Theorem \ref{thm2} and the previous result, a complete classification of finite abelian groups with respect to their Beauville dimension easily follows.

\begin{cor}
\label{cor1}
Let $G$ be a finite abelian group. Then, $d(G) \in \{ 1, 2, 4 \}$. In particular,\\
- if $d(G) = 1$, then $G$ is not a generalised Beauville group;\\
- if $d(G) = 2$, then $G$ is a Beauville group;\\
- if $d(G) = 4$, then $G$ is a generalised Beauville group of the form $C_n \times C_n \times H$, where $n$ is a power of $3$ and $H$ is either the trivial group or an abelian Beauville group.
\end{cor}

From the computational results given in the next section, we were also able to find a number of non-abelian infinite families of generalised Beauville groups. The following is a typical example.

\begin{thm}
\label{thm8}
Let $G$ be a finite group with structure $C_3 \times (C_p : C_3)$, where $p$ is a prime and $p \equiv 1 \ ({\emph{mod }}  3)$. Then, $d(G) = 4$.
\end{thm}

\begin{proof}
The group $G$ contains exactly $p - 1$ elements of order $p$, $2 (p -1)$ elements of order $3 p$, and $2 (3 p + 1)$ elements of order $3$. There are in total $p + 8$ conjugacy classes: $1$ for the identity element, $1$ for $z \in G$ such that $Z(G) = \langle z \rangle$, $1$ for $z^2$, $6$ each containing $p$ non-central elements of order $3$,  $(p -1) / 3$ with $3$ elements of order $p$ each, and $2(p -1) / 3$ each with $3$ elements of order $3 p$. We now show that $G$ can only have precisely 4 sets $\Sigma$:
\begin{enumerate}
\item $\Sigma_1$ contains all the elements of orders $p$ and $3p$, and $2(2p + 1)$ elements of order 3;
\item $\Sigma_2$ contains all the elements of orders $p$ and $3p$, and $2(2p + 1)$ elements of order 3;
\item $\Sigma_3$ contains all the elements of orders $p$ and $3p$, and $2(2p + 1)$ elements of order 3;
\item $\Sigma_4$ contains $6p$ elements of order 3.
\end{enumerate}
Looking at all pairs of generators $g_1, g_2$ for $G$, there are only two possible combinations:
\begin{enumerate}[label=\alph*)]
\item $g_1, g_2$ have both order $3$ and do not belong to $Z(G)$, in which case $\Sigma(g_1, g_2)$ contains 6 conjugacy classes of elements of order 3 ($6p$ elements), but no element from $Z(G)$, and hence $\Sigma(g_1, g_2) = \Sigma_4$;
\item without loss of generality, $g_1$ has order 3 and does not belong to $Z(G)$, while $g_2$ has order $3p$, so that $\Sigma(g_1, g_2)$ contains the $7p -1$ elements given by all the conjugacy classes with elements of order $p$ and $3p$, $Z(G)$ (elements of order $3p$ power up to the centre), and 4 conjugacy classes of generating elements of order 3, and thus there can only be 3 such sets, {\it{i.e.}}, $\Sigma(g_1, g_2) \in \{ \Sigma_1, \Sigma_2, \Sigma_3 \}$.
\end{enumerate}
Having only 4 sets $\Sigma$, $d(G) \leq 4$. Since each set $\Sigma$ contains at least 4 of the 6 available conjugacy classes of generating elements of order 3 and $Z(G) \in \bigcap_{1 \leq i \leq 3} \Sigma_i $, $d(G) \notin \{ 2, 3 \}$. However, $\bigcap_{1 \leq i \leq 4}  \Sigma_i = \emptyset$, so $d(G) \neq 1$.
Therefore, $d(G) = 4$.
\end{proof}

Notice that in the above result we need the condition $p \equiv 1 \ ({\text{mod }}  3)$ as we can find a non-trivial homomorphisms to construct a semidirect product of the form  $C_n : C_m$ if and only if $\gcd(\phi(n), m) > 1$. We also want $p$ to be prime, because otherwise, as in the case of $C_3 \times (C_{70} : C_3)$, we might get a group with $d = 1$. Finally, observe that substituting $3$ with any other odd prime, we obtain a Beauville group.\\
With very similar arguments and techniques, it is not too difficult to show that , for $k, n \in \nn$ and primes $p_1, p_2 \equiv 1 \ ({\text{mod }}  3)$, groups with structures $A_4 \times (C_{p_1} : C_3)$, $(C_{p_1} \times C_3 \times C_3) : C_3)$, $(C_{3 k} \times C_{3 k}) : C_3$, $C_3 \times ((C_n \times C_n) : C_3)$, $C_9 \times ((C_n \times C_n) : C_9)$, $(C_{p_1} : C_3) \times (C_{p_2} : C_3)$, $C_3 \times ((C_{p_1} \times C_2 \times C_2) : C_3)$ have all $d = 4$.

\newpage
\section{Generalised Beauville Groups (orders 1-1023)}
\label{sec3}
Relying on the results from the previous section, those proved in \cite{Bar, Wei}, and our GAP algorithms, we have scanned groups of size between 1 and 1023 (11759892 groups) searching for generalised Beauville groups. Notice that the algorithm used in \cite{Bar} was specifically designed for Beauville $p$-groups: it checked only for intersections between any {\it{two}} sets $\Sigma$ by exploiting the way in which generators of $p$-groups are stored in GAP; hence, it would have not been able to spot higher values of $d$ or Beauville non-$p$-groups. The full list we have obtained is presented in the following tables (arranged by the values of $d$). Given any generalised Beauville group $G$, the first column displays its coordinates in the Small Groups library (see \cite{sma}), while the second provides its structure as saved in GAP.

\begin{minipage}{.5\linewidth}
\scalebox{0.71}{
\begin{tabular}{|l|l|}
\hline
\multicolumn{2}{|c|}{Generalised Beauville Groups with $d = 2$}\\
\hline
SmallGroup & Structure Description\\
\hline
$(25, 2)$ & $C_{5} \times C_{5}$\\
\hline
$(49, 2)$ & $C_7 \times C_7$\\
\hline
$(120, 34)$ & $S_5$\\
\hline
$(121, 2)$ & $C_{11} \times C_{11}$\\
\hline
$(125, 3)$ & $(C_5 \times C_5) : C_5$\\
\hline
$(128, 36)$ & $(C_2 \times ((C_4 \times C_2) : C_2)) : C_4$\\
\hline
$(168, 42)$ & $PSL_{3}(2)$\\
\hline
$(169, 2)$ & $C_{13} \times C_{13}$\\
\hline
$(240, 90)$ & $SL_{2}(5) : C_2$\\
\hline
$(240, 91)$ & $A_5 : C_4$\\
\hline
$(240, 189)$ & $C_2 \times S_5$\\
\hline
$(243, 3)$ & $(C_3 \times ((C_3 \times C_3) : C_3)) : C_3$\\
\hline
$(256, 295)$ & $(C_2 \times (((C_4 \times C_2) : C_2) : C_2)) : C_4$\\
\hline
$(256, 298)$ & $(C_2 \times (((C_4 \times C_2) : C_2) : C_2)) : C_4$\\
\hline
$(256, 306)$ & $(((C_4 \times C_2) : C_4) : C_2) : C_4$\\
\hline
$(275, 3)$ & $C_5 \times (C_{11} : C_5)$\\
\hline
$(289, 2)$ & $C_{17} \times C_{17}$\\
\hline
$(300, 22)$ & $C_5 \times A_5$\\
\hline
$(320, 1635)$ & $((C_2 \times C_2 \times C_2 \times C_2) : C_5) : C_4$\\
\hline
$(324, 160)$ & $((C_3 \times ((C_3 \times C_3) : C_2)) : C_2) : C_3$\\
\hline
$(336, 114)$ & $SL_{2}(7)$\\
\hline
$(336, 208)$ & $PSL_{3}(2) : C_2$\\
\hline
$(336, 209)$ & $C_2 \times PSL_{3}(2)$\\
\hline
$(343, 3)$ & $(C_7 \times C_7) : C_7$\\
\hline
$(360, 118)$ & $A_6$\\
\hline
$(360, 119)$ & $C_3 \times S_5$\\
\hline
$(360, 120)$ & $GL_{2}(4) : C_2$\\
\hline
$(361, 2)$ & $C_{19} \times C_{19}$\\
\hline
$(392, 39)$ & $C_7 \times ((C_2 \times C_2 \times C_2) : C_7)$\\
\hline
$(400, 213)$ & $C_5 \times ((C_2 \times C_2 \times C_2 \times C_2) : C_5)$\\
\hline
$(480, 218)$ & $GL_{2}(5)$\\
\hline
$(480, 219)$ & $SL_{2}(5) : C_4$\\
\hline
$(480, 948)$ & $(SL_{2}(5) : C_2) : C_2$\\
\hline
$(480, 950)$ & $C_2 \times (SL_{2}(5) : C_2)$\\
\hline
$(480, 951)$ & $(C_2 \times C_2 \times A_5) : C_2$\\
\hline
$(480, 952)$ & $C_2 \times (A_5 : C_4)$\\
\hline
$(504, 156)$ & $PSL_{2}(8)$\\
\hline
$(504, 157)$ & $C_3 \times PSL_{3}(2)$\\
\hline
$(512, 325)$ & $(((C_4 \times C_2 \times C_2) : C_4) : C_2) : C_4$\\
\hline
$(512, 335)$ & $((C_2 \times (((C_4 \times C_2) : C_2) : C_2)) : C_2) : C_4$\\
\hline
$(512, 351)$ & $(C_2 \times ((C_4 \times C_2 \times C_2) : C_4)) : C_4$\\
\hline
$(512, 1572)$ & $((C_4 \times C_4) : C_8) : C_4$\\
\hline
$(512, 1574)$ & $(C_4 \times (C_8 : C_4)) : C_4$\\
\hline
$(512, 1632)$ & $(C_2 \times ((((C_4 \times C_2) : C_2) : C_2) : C_2)) : C_4$\\
\hline
$(512, 1634)$ & $(C_2 \times ((C_4 \times C_4) : C_4)) : C_4$\\
\hline
$(512, 1637)$ & $(C_2 \times ((C_4 \times C_4) : C_4)) : C_4$\\
\hline
$(512, 1641)$ & $((C_2 . (((C_4 \times C_2) : C_2) : C_2)$\\
\hline
$(512, 1642)$ & $(((C_4 \times C_4) : C_4) : C_2) : C_4$\\
\hline
$(512, 1643)$ & $((C_2 \times ((C_8 : C_2) : C_2)) : C_2) : C_4$\\
\hline
$(512, 1644)$ & $((C_2 \times (((C_4 \times C_2) : C_2) : C_2)) : C_2) : C_4$\\
\hline
$(512, 1649)$ & $(((C_8 \times C_2) : C_4) : C_2) : C_4$\\
\hline
$(529, 2)$ & $C_{23} \times C_{23}$\\
\hline
$(576, 8652)$ & $(A_4 \times A_4) : C_4$\\
\hline
$(600, 54)$ & $C_5 \times SL_{2}(5)$\\
\hline
$(600, 145)$ & $(C_5 \times A_5) : C_2$\\
\hline
$(625, 2)$ & $C_{25} \times C_{25}$\\
\hline
$(625, 4)$ & $C_{25} : C_{25}$\\
\hline
$(625, 7)$ & $(C_5 \times C_5 \times C_5) : C_5$\\
\hline
\end{tabular}
}
\end{minipage}
\begin{minipage}{.5\linewidth}
\scalebox{0.71}{
\begin{tabular}{|l|l|}
\hline
\multicolumn{2}{|c|}{Generalised Beauville Groups with $d = 2$}\\
\hline
SmallGroup & Structure Description\\
\hline
$(640, 787)$ & $(C_{10} \times ((C_4 \times C_2) : C_2)) : C_4$\\
\hline
$(640, 21454)$ & $C_2 . (((C_2 \times C_2 \times C_2 \times C_2) : C_5) : C_4)$\\
\hline
$(640, 21455)$ & $(((C_2 \times Q_8) : C_2) : C_5) : C_4$\\
\hline
$(640, 21456)$ & $((C_2 \times C_2 \times C_2 \times C_2) : C_5) : C_8$\\
\hline
$(640, 21536)$ & $C_2 \times (((C_2 \times C_2 \times C_2 \times C_2) : C_5) : C_4)$\\
\hline
$(660, 13)$ & $PSL_{2}(11)$\\
\hline
$(672, 1043)$ & $PSL_{3}(2) : C_4$\\
\hline
$(672, 1044)$ & $SL_{2}(7) : C_2$\\
\hline
$(672, 1046)$ & $C_4 \times PSL_{3}(2)$\\
\hline
$(672, 1048)$ & $C_2 \times SL_{2}(7)$\\
\hline
$(672, 1254)$ & $C_2 \times (PSL_{3}(2) : C_2)$\\
\hline
$(672, 1255)$ & $C_2 \times C_2 \times PSL_{3}(2)$\\
\hline
$(720, 409)$ & $SL_{2}(9)$\\
\hline
$(720, 411)$ & $C_3 \times (SL_{2}(5) : C_2)$\\
\hline
$(720, 412)$ & $C_3 \times (A_5 : C_4)$\\
\hline
$(720, 413)$ & $GL_{2}(4) : C_4$\\
\hline
$(720, 415)$ & $(C_3 \times SL_{2}(5)) : C_2$\\
\hline
$(720, 763)$ & $S_6$\\
\hline
$(720, 764)$ & $A_6 : C_2$\\
\hline
$(720, 766)$ & $C_2 \times A_6$\\
\hline
$(720, 767)$ & $S_5 \times S_3$\\
\hline
$(720, 769)$ & $C_6 \times S_5$\\
\hline
$(720, 770)$ & $C_2 \times (GL_{2}(4) : C_2)$\\
\hline
$(729, 34)$ & $((C_3 \times C_3 \times C_3 \times C_3) : C_3) : C_3$\\
\hline
$(729, 37)$ & $(C_3 \times ((C_9 \times C_3) : C_3)) : C_3$\\
\hline
$(729, 40)$ & $(C_3 \times ((C_9 \times C_3) : C_3)) : C_3$\\
\hline
$(775, 3)$ & $C_5 \times (C_{31} : C_5)$\\
\hline
$(800, 1065)$ & $C_5 \times (((C_2 \times Q_8) : C_2) : C_5)$\\
\hline
$(841, 2 )$ & $C_{29} \times C_{29}$\\
\hline
$(864, 2666)$ & $((C_2 \times ((C_3 \times C_3) : C_4)) : C_4) : C_3$\\
\hline
$(864, 4445)$ & $C_6 \times S_3 \times SL_{2}(3)$\\
\hline
$(864, 4665)$ & $C_2 \times C_2 \times (((C_3 \times C_3) : Q_8) : C_3)$\\
\hline
$(864, 4680)$ & $A_4 \times A_4 \times S_3$\\
\hline
$(900, 88)$ & $C_{15} \times A_5$\\
\hline
$(960, 639)$ & $SL_{2}(5) : C_8$\\
\hline
$(960, 5693)$ & $C_2 \times GL_{2}(5)$\\
\hline
$(960, 5694)$ & $GL_{2}(5) : C_2$\\
\hline
$(960, 5699)$ & $(C_4 \times SL_{2}(5)) : C_2$\\
\hline
$(960, 5704)$ & $(C_4 \times SL_{2}(5)) : C_2$\\
\hline
$(960, 5709)$ & $(C_2 \times (SL_{2}(5) : C_2)) : C_2$\\
\hline
$(960, 5710)$ & $(SL_{2}(5) : C_4) : C_2$\\
\hline
$(960, 5711)$ & $(SL_{2}(5) : C_4) : C_2$\\
\hline
$(960, 5712)$ & $SL_{2}(5) : Q_8$\\
\hline
$(960, 5716)$ & $((SL_{2}(5) : C_2) : C_2) : C_2$\\
\hline
$(960, 5719)$ & $((SL_{2}(5) : C_2) : C_2) : C_2$\\
\hline
$(960, 5721)$ & $(C_2 \times (A_5 : C_4)) : C_2$\\
\hline
$(960, 5723)$ & $(C_2 \times C_2 \times SL_{2}(5)) : C_2$\\
\hline
$(960, 5724)$ & $C_2 \times (SL_{2}(5) : C_4)$\\
\hline
$(960, 5725)$ & $(SL_{2}(5) : C_4) : C_2$\\
\hline
$(961, 2)$ & $C_{31} \times C_{31}$\\
\hline
$(972, 138)$ & $(C_2 \times C_2 \times ((C_9 \times C_3) : C_3)) : C_3$\\
\hline
$(972, 757)$ & $((C_3 \times ((C_3 \times C_3) : C_2)) : C_2) : C_9$\\
\hline
$(972, 877)$ & $C_3 \times (((C_3 \times ((C_3 \times C_3) : C_2)) : C_2) : C_3)$\\
\hline
$(1008, 517)$ & $C_3 \times SL_{2}(7)$\\
\hline
$(1008, 881)$ & $(C_3 \times PSL_{3}(2)) : C_2$\\
\hline
$(1008, 882)$ & $C_3 \times (PSL_{3}(2) : C_2)$\\
\hline
$(1008, 883)$ & $S_3 \times PSL_{3}(2)$\\
\hline
$(1008, 884)$ & $C_6 \times PSL_{3}(2)$\\
\hline
\end{tabular}
}
\end{minipage}

\begin{minipage}[t]{.5\linewidth}
\scalebox{0.71}{
\begin{tabular}[t]{|l|l|}
\hline
\multicolumn{2}{|c|}{Generalised Beauville Groups with $d = 3$}\\
\hline
SmallGroup & Structure Description\\
\hline
$(243, 4)$ & $(C_3 \times (C_9 : C_3)) : C_3$\\
\hline
$(243, 13)$ & $((C_9 \times C_3) : C_3) : C_3$\\
\hline
$(432, 623)$ & $C_3 \times S_3 \times SL_{2}(3)$\\
\hline
$(432, 735)$ & $C_2 \times (((C_3 \times C_3) : Q_8) : C_3)$\\
\hline
$(432, 749)$ & $A_4 \times S_3 \times S_3$\\
\hline
$(432, 763)$ & $C_6 \times A_4 \times S_3$\\
\hline
$(729, 9)$ & $(C_3 \times ((C_9 \times C_3) : C_3)) : C_3$\\
\hline
$(729, 35)$ & $((C_3 \times (C_9 : C_3)) : C_3) : C_3$\\
\hline
$(729, 38)$ & $(C_3 \times (C_3 . ((C_3 \times C_3) : C_3)$\\
\hline
$(729, 41)$ & $(C_3 \times (C_3 . ((C_3 \times C_3) : C_3)$\\
\hline
$(729, 44)$ & $((C_9 : C_9) : C_3) : C_3$\\
\hline
$(729, 45)$ & $((C_9 : C_9) : C_3) : C_3$\\
\hline
$(729, 48)$ & $((C_9 \times C_3) : C_9) : C_3$\\
\hline
$(729, 65)$ & $(((C_9 \times C_3) : C_3) : C_3) : C_3$\\
\hline
$(729, 68)$ & $(((C_9 \times C_3) : C_3) : C_3) : C_3$\\
\hline
$(729, 87)$ & $(C_3 \times (C_9 : C_9)) : C_3$\\
\hline
$(864, 4186)$ & $(C_3 \times S_3 \times SL_{2}(3)) : C_2$\\
\hline
$(864, 4187)$ & $(C_3 \times S_3 \times SL_{2}(3)) : C_2$\\
\hline
$(864, 4446)$ & $C_3 \times ((C_6 \times SL_{2}(3)) : C_2)$\\
\hline
$(864, 4449)$ & $C_3 \times ((S_3 \times SL{2}(3)) : C_2)$\\
\hline
$(864, 4708)$ & $C_3 \times C_3 \times (((C_2 \times C_2 \times C_2 \times C_2) : C_3) : C_2)$\\
\hline
$(960, 5696)$ & $C_4 \times (SL_{2}(5) : C_2)$\\
\hline
$(960, 5707)$ & $(C_2 \times C_2) . (C_2 \times S_5)$\\
\hline
$(972, 135)$ & $(C_2 \times C_2 \times ((C_9 \times C_3) : C_3)) : C_3$\\
\hline
$(972, 141)$ & $(C_2 \times C_2 \times ((C_9 \times C_3) : C_3)) : C_3$\\
\hline
$(972, 179)$ & $(C_2 \times C_6 \times ((C_3 \times C_3) : C_3)) : C_3$\\
\hline
$(972, 183)$ & $(C_2 \times C_6 \times (C_9 : C_3)) : C_3$\\
\hline
\end{tabular}
}
\end{minipage}
\begin{minipage}[t]{.5\linewidth}
\scalebox{0.71}{
\begin{tabular}[t]{|l|l|}
\hline
\multicolumn{2}{|c|}{Generalised Beauville Groups with $d = 4$}\\
\hline
SmallGroup & Structure Description\\
\hline
$(9, 2)$ & $C_{3} \times C_{3}$\\
\hline
$(27, 3)$ & $(C_3 \times C_3) : C_3$\\
\hline
$(36, 11)$ & $C_3 \times A_4$\\
\hline
$(63, 3)$ & $C_3 \times (C_7 : C_3)$\\
\hline
$(81, 2)$ & $C_9 \times C_9$\\
\hline
$(81, 4)$ & $C_9 : C_9$\\
\hline
$(81, 9)$ & $(C_9 \times C_3) : C_3$\\
\hline
$(108, 22)$ & $(C_6 \times C_6) : C_3$\\
\hline
$(117, 3)$ & $C_3 \times (C_{13} : C_3)$\\
\hline
$(144, 68)$ & $C_3 \times ((C_4 \times C_4) : C_3)$\\
\hline
$(144, 156)$ & $C_6 \times SL_{2}(3)$\\
\hline
$(144, 184)$ & $A_4 \times A_4$\\
\hline
$(144, 193)$ & $C_2 \times C_6 \times A_4$\\
\hline
$(171, 4)$ & $C_3 \times (C_{19} : C_3)$\\
\hline
$(189, 8)$ & $(C_{21} \times C_3) : C_3$\\
\hline
$(225, 5)$ & $C_3 \times ((C_5 \times C_5) : C_3)$\\
\hline
$(225, 6)$ & $C_{15} \times C_{15}$\\
\hline
$(243, 2)$ & $(C_9 \times C_3) : C_9$\\
\hline
$(243, 9)$ & $(C_3 \times C_3) . ((C_3 \times C_3) : C_3)$\\
\hline
$(243, 14)$ & $(C_9 \times C_3) : C_9$\\
\hline
$(243, 26)$ & $(C_9 \times C_9) : C_3$\\
\hline
$(243, 28)$ & $(C_9 : C_9) : C_3$\\
\hline
$(252, 27)$ & $A_4 \times (C_7 : C_3)$\\
\hline
$(252, 40)$ & $C_3 \times ((C_{14} \times C_2) : C_3)$\\
\hline
$(279, 3)$ & $C_3 \times (C_{31} : C_3)$\\
\hline
$(288, 230)$ & $C_3 \times (((C_4 \times C_2) : C_4) : C_3)$\\
\hline
$(288, 859)$ & $A_4 \times SL_{2}(3)$\\
\hline
$(288, 981)$ & $C_2 \times C_6 \times SL_{2}(3)$\\
\hline
$(288, 985)$ & $C_3 \times ((C_2 \times SL_{2}(3)) : C_2)$\\
\hline
$(288, 1029)$ & $C_2 \times A_4 \times A_4$\\
\hline
$(324, 46 )$ & $C_9 \times ((C_2 \times C_2) : C_9)$\\
\hline
$(324, 47)$ & $((C_2 \times C_2) : C_9) : C_9$\\
\hline
$(324, 48)$ & $(C_{18} \times C_2) : C_9$\\
\hline
$(324, 50)$ & $(C_{18} \times C_6) : C_3$\\
\hline
$(324, 54)$ & $(C_2 \times C_2 \times ((C_3 \times C_3) : C_3)) : C_3$\\
\hline
$(333, 4)$ & $C_3 \times (C_{37} : C_3)$\\
\hline
$(351, 8)$ & $(C_{39} \times C_3) : C_3$\\
\hline
$(387, 3)$ & $C_3 \times (C_{43} : C_3)$\\
\hline
$(432, 103)$ & $(C_{12} \times C_{12}) : C_3$\\
\hline
$(432, 336)$ & $C_2 \times ((C_3 \times C_3 \times Q_8) : C_3)$\\
\hline
$(432, 526)$ & $(C_2 \times C_6 \times A_4) : C_3$\\
\hline
$(432, 550)$ & $C_2 \times C_2 \times ((C_6 \times C_6) : C_3)$\\
\hline
$(441, 3)$ & $C_3 \times (C_{49} : C_3)$\\
\hline
$(441, 9)$ & $(C_7 : C_3) \times (C_7 : C_3)$\\
\hline
$(441, 12)$ & $C_3 \times ((C_7 \times C_7) : C_3)$\\
\hline
$(441, 13)$ & $C_{21} \times C_{21}$\\
\hline
$(468, 32)$ & $A_4 \times (C_{13} : C_3)$\\
\hline
$(468, 49)$ & $C_3 \times ((C_{26} \times C_2) : C_3)$\\
\hline
$(504, 158)$ & $C_3 \times (((C_2 \times C_2 \times C_2) : C_7) : C_3)$\\
\hline
$(513, 9)$ & $(C_{57} \times C_3) : C_3$\\
\hline
$(549, 3)$ & $C_3 \times (C_{61} : C_3)$\\
\hline
$(567, 8)$ & $C_9 \times (C_7 : C_9)$\\
\hline
$(567, 9)$ & $(C_7 : C_9) : C_9$\\
\hline
$(567, 10)$ & $C_{63} : C_9$\\
\hline
$(567, 11)$ & $C_{63} : C_9$\\
\hline
$(567, 13)$ & $(C_{63} \times C_3) : C_3$\\
\hline
$(567, 17)$ & $(C_3 \times (C_7 : C_9)) : C_3$\\
\hline
\end{tabular}
}
\end{minipage}

\begin{minipage}[t]{.5\textwidth}
\scalebox{0.71}{
\begin{tabular}[t]{|l|l|}
\hline
\multicolumn{2}{|c|}{Generalised Beauville Groups with $d = 4$}\\
\hline
$(576, 1070)$ & $C_3 \times ((C_8 \times C_8) : C_3)$\\
\hline
$(576, 3609)$ & $C_6 \times (((C_4 \times C_2) : C_4) : C_3)$\\
\hline
$(576, 3615)$ & $C_3 \times (((C_2 \times C_2 \times Q_8) : C_3) : C_2)$\\
\hline
$(576, 3617)$ & $C_3 \times ((((C_4 \times C_2) : C_4) : C_3) : C_2)$\\
\hline
$(576, 3618)$ & $C_3 \times ((((C_4 \times C_2) : C_4) : C_3) : C_2)$\\
\hline
$(576, 3621)$ & $C_3 \times ((((C_2 \times D_8) : C_2) : C_3) : C_2)$\\
\hline
$(576, 5127)$ & $A_4 \times ((C_4 \times C_4) : C_3)$\\
\hline
$(576, 5128)$ & $SL_{2}(3) \times SL_{2}(3)$\\
\hline
$(576, 5129)$ & $((((C_2 \times C_2 \times C_2 \times C_2) : C_3) : C_2) : C_2) : C_3$\\
\hline
$(576, 7412)$ & $C_2 \times C_6 \times ((C_4 \times C_4) : C_3)$\\
\hline
$(576, 7417)$ & $C_{12} \times (SL_{2}(3) : C_2)$\\
\hline
$(576, 7418)$ & $C_3 \times ((C_4 \times SL_{2}(3)) : C_2)$\\
\hline
$(576, 7420)$ & $C_6 \times (((C_2 \times C_2 \times C_2 \times C_2) : C_3) : C_2)$\\
\hline
$(576, 7421)$ & $C_6 \times (((C_4 \times C_4) : C_3) : C_2)$\\
\hline
$(576, 7422)$ & $C_6 \times (((C_4 \times C_4) : C_3) : C_2)$\\
\hline
$(576, 7423)$ & $C_3 \times ((C_2 \times C_2 \times SL_{2}(3)) : C_2)$\\
\hline
$(576, 7424)$ & $C_3 \times D_8 \times SL_{2}(3)$\\
\hline
$(576, 7428)$ & $C_3 \times ((((C_4 \times C_4) : C_3) : C_2) : C_2)$\\
\hline
$(576, 7429)$ & $C_3 \times ((((C_2 \times C_2 \times C_2 \times C_2) : C_3) : C_2) : C_2)$\\
\hline
$(576, 8357)$ & $C_2 \times A_4 \times SL_{2}(3)$\\
\hline
$(576, 8360)$ & $((((C_2 \times D_8) : C_2) : C_3) : C_3) : C_2$\\
\hline
$(576, 8662)$ & $C_2 \times C_2 \times A_4 \times A_4$\\
\hline
$(576, 8664)$ & $(C_2 \times C_2 \times ((C_2 \times C_2 \times C_2 \times C_2) : C_3)) : C3$\\
\hline
$(603, 3)$ & $C_3 \times (C_{67} : C_3)$\\
\hline
$(657, 4)$ & $C_3 \times (C_{73} : C_3)$\\
\hline
$(675, 12)$ & $(C_{15} \times C_{15}) : C_3$\\
\hline
$(675, 14)$ & $C_5 \times C_5 \times ((C_3 \times C_3) : C_3)$\\
\hline
$(684, 32)$ & $A_4 \times (C_{19} : C_3)$\\
\hline
$(684, 45)$ & $C_3 \times ((C_{38} \times C_2) : C_3)$\\
\hline
$(711, 3)$ & $C_3 \times (C_{79} : C_3)$\\
\hline
$(729, 2)$ & $C_{27} \times C_{27}$\\
\hline
$(729, 3)$ & $C_{27} : C_{27}$\\
\hline
$(729, 10)$ & $((C_9 \times C_3) : C_3) : C_9$\\
\hline
$(729, 12)$ & $(C_3 \times (C_9 : C_3)) : C_9$\\
\hline
$(729, 22)$ & $C_{27} : C_{27}$\\
\hline
$(729, 24)$ & $(C_9 \times C_9) : C_9$\\
\hline
$(729, 26)$ & $(C_9 \times C_3 \times C_3) : C_9$\\
\hline
$(729, 30)$ & $(C_9 \times C_9) : C_9$\\
\hline
$(729, 46)$ & $C_3 . ((C_3 \times (C_9 : C_3)) : C_3)$\\
\hline
$(729, 47)$ & $C_3 . ((C_3 \times (C_9 : C_3)) : C_3)$\\
\hline
$(729, 50)$ & $(C_3 \times C_3) . ((C_9 \times C_3) : C_3)$\\
\hline
$(729, 52)$ & $(C_3 \times C_3) . ((C_9 \times C_3) : C_3)$\\
\hline
$(729, 56)$ & $C_3 . ((C_3 \times C_3) . ((C_3 \times C_3) : C_3)$\\
\hline
$(729, 57)$ & $C_3 . ((C_3 \times C_3) . ((C_3 \times C_3) : C_3)$\\
\hline
$(729, 66)$ & $(C_9 \times C_3 \times C_3) : C_9$\\
\hline
$(729, 69)$ & $(C_3 \times (C_9 : C_3)) : C_9$\\
\hline
$(729, 72)$ & $(C_9 \times C_9) : C_9$\\
\hline
$(729, 73)$ & $(C_9 \times C_9) : C_9$\\
\hline
$(729, 75)$ & $(C_9 : C_9) : C_9$\\
\hline
$(729, 78)$ & $(C_9 \times C_9) : C_9$\\
\hline
$(729, 85)$ & $(C_3 . ((C_3 \times C_3) : C_3)$\\
\hline
$(729, 89)$ & $(C_3 . ((C_3 \times C_3) : C_3)$\\
\hline
$(729, 95)$ & $(C_{27} \times C_9) : C_3$\\
\hline
$(756, 64)$ & $(C_2 \times C_6 \times (C_7 : C_3)) : C_3$\\
\hline
$(756, 117)$ & $(C_{42} \times C_6) : C_3$\\
\hline
$(819, 6)$ & $(C_7 : C_3) \times (C_{13} : C_3)$\\
\hline
$(819, 9)$ & $C_3 \times (C_{91} : C_3)$\\
\hline
\end{tabular}
}
\end{minipage}
\begin{minipage}[t]{.5\textwidth}
\scalebox{0.71}{
\begin{tabular}[t]{|l|l|}
\hline
\multicolumn{2}{|c|}{Generalised Beauville Groups with $d = 4$}\\
\hline
$(819, 10)$ & $C_3 \times (C_{91} : C_3)$\\
\hline
$(837, 8 )$ & $(C_{93} \times C_3) : C_3$\\
\hline
$(864, 307)$ & $(C_3 \times C_3 \times ((C_4 \times C_2) : C_4)) : C_3$\\
\hline
$(864, 2245)$ & $(C_2 \times C_6 \times SL_{2}(3)) : C_3$\\
\hline
$(864, 2537)$ & $C_2 \times C_2 \times ((C_3 \times C_3 \times Q_8) : C_3)$\\
\hline
$(864, 2547)$ & $(C_2 \times ((C_3 \times C_3 \times Q_8) : C_3)) : C_2$\\
\hline
$(864, 4003)$ & $C_2 \times ((C_2 \times C_6 \times A_4) : C_3)$\\
\hline
$(864, 4189)$ & $S_3 \times S_3 \times SL_{2}(3)$\\
\hline
$(873, 3)$ & $C_3 \times (C_{97} : C_3)$\\
\hline
$(900, 98)$ & $A_4 \times ((C_5 \times C_5) : C_3)$\\
\hline
$(900, 140)$ & $C_5 \times C_{15} \times A_4$\\
\hline
$(900, 141)$ & $C_3 \times ((C_{10} \times C_{10}) : C_3)$\\
\hline
$(927, 3)$ & $C_3 \times (C_{103} : C_3)$\\
\hline
$(972, 122)$ & $(C_{18} \times C_{18}) : C_3$\\
\hline
$(972, 128)$ & $(C_2 \times C_2) : ((C_3 \times C_3) . ((C_3 \times C_3) : C_3)$\\
\hline
$(972, 131)$ & $(C_3 \times ((C_2 \times C_2) : C_9)) : C_9$\\
\hline
$(972, 143)$ & $(C_2 \times C_2 \times (C_9 : C_9)) : C_3$\\
\hline
$(972, 159)$ & $(C_2 \times C_2 \times ((C_9 \times C_3) : C_3)) : C_3$\\
\hline
$(972, 160)$ & $(C_2 \times C_2 \times ((C_9 \times C_3) : C_3)) : C_3$\\
\hline
$(972, 164)$ & $(C_2 \times C_2 \times (C_3 . ((C_3 \times C_3) : C_3)$\\
\hline
$(972, 171)$ & $(C_{18} \times C_6) : C_9$\\
\hline
$(972, 173)$ & $(C_{18} \times C_6) : C_9$\\
\hline
$(972, 176)$ & $(C_3 \times (((C_2 \times C_2) : C_9) : C_3)) : C_3$\\
\hline
$(972, 186)$ & $((C_3 \times ((C_2 \times C_2) : C_9)) : C_3) : C_3$\\
\hline
$(981, 4)$ & $C_3 \times (C_{109} : C_3)$\\
\hline
$(999, 9)$ & $(C_{111} \times C_3) : C_3$\\
\hline
$(1008, 242)$ & $(C_7 : C_3) \times ((C_4 \times C_4) : C_3)$\\
\hline
$(1008, 409)$ & $C_3 \times ((C_{28} \times C_4) : C_3)$\\
\hline
$(1008, 525)$ & $((C_7 : C_3) : C_2) \times SL_{2}(3)$\\
\hline
$(1008, 555)$ & $C_2 \times ((C_7 : C_3) \times SL_{2}(3))$\\
\hline
$(1008, 671)$ & $C_3 \times (((C_7 \times Q_8) : C_3) : C_2)$\\
\hline
$(1008, 824)$ & $C_6 \times ((C_7 \times Q_8) : C_3)$\\
\hline
$(1008, 887)$ & $C_6 \times (((C_2 \times C_2 \times C_2) : C_7) : C_3)$\\
\hline
$(1008, 890)$ & $C_2 \times (A_4 \times ((C_7 : C_3) : C_2))$\\
\hline
$(1008, 909)$ & $C_2 \times C_2 \times (A_4 \times (C_7 : C_3))$\\
\hline
$(1008, 912)$ & $A_4 \times ((C_{14} \times C_2) : C_3)$\\
\hline
$(1008, 913)$ & $(C_2 \times C_2 \times ((C_{14} \times C_2) : C_3)) : C_3$\\
\hline
$(1008, 932)$ & $C_6 \times (((C_{14} \times C_2) : C_3) : C_2)$\\
\hline
$(1008, 946)$ & $C_2 \times C_6 \times ((C_{14} \times C_2) : C_3)$\\
\hline
\end{tabular}
}
\end{minipage}

\newpage
\section{Conclusion}
\label{sec4}

The groups displayed in the above tables have brought a significant number of patterns to our attention and opened a lot of interesting questions. In particular, there are several problems regarding non-abelian generalised Beauville groups which still need to be solved. The following list represents an attempt to capture the main ones.

\begin{problem}
\label{prob1}
Can we have a finite group $G$ with $d(G) > 4$?
\end{problem}

\begin{problem}
\label{prob2}
If the answer to Problem \ref{prob1} is yes, is there a maximum value for $d$?
\end{problem}

Considering most of the examples given in this note, the obvious approach to groups with $d > 2$ would be to try and prove that they must contain a copy of $C_3 \times C_3$, whose structures can either be lifted, or play a nice role as in Theorem \ref{thm8}. However, as we can see in one of the tables in Section \ref{sec3}, groups such as SmallGroup(960, 5696) have Beauville dimension $d = 3$ even though their order is not divisible by 9. Yet, in our computations we could not find any generalised Beauville group with $d > 2$ whose order is not divisible by 3. Hence, we are still left with the next two problems.

\begin{problem}
\label{prob4}
Can we have $d > 2$ for finite groups whose order is not divisible by $3$?
\end{problem}

\begin{problem}
\label{prob5}
If the answer to Problem \ref{prob4} is no, can we explain why that would be the case?
\end{problem}

Wider research perspectives which are currently under exploration include:
\begin{enumerate}
\item asymptotic results on the distribution of generalised Beauville groups similar to those obtained in \cite[Corollary 1.5]{Bar};
\item automorphisms of generalised Beauville varieties (see \cite{Jon1});
\item additional properties of generalised Beauville structures such as strongly real (for $i = 1, 2$, we say that a Beauville group $G$ and its structure ${\cal{X}} = \{ x_i, y_i \}$ are strongly real if  there exists an automorphism $\phi \in {\text{Aut}}(G)$ and elements $g_i \in G$ such that  $g_i \phi(x_i) g_i ^{-1} = x_i ^{-1}$ and $g_i \phi(y_i) g_i ^{-1} = y_i ^{-1}$), or mixed (where the curves are transposed by elements of the group and more specific restrictions apply).
\end{enumerate}

\end{document}